\documentclass[12pt,a4paper]{article}
\usepackage[a4paper,top=2.3cm,bottom=2.3cm,left=2.3cm,right=2.3cm]{geometry}
\usepackage{amsmath}
 \usepackage{graphicx}
\usepackage{amsthm}
\usepackage{color}
\usepackage{fancyhdr}
\usepackage{pdfpages}
\usepackage{amssymb}
\usepackage{paralist}
\usepackage{extarrows}
\usepackage[colorlinks=true,citecolor=blue,linkcolor=blue,anchorcolor=blue,urlcolor=blue]{hyperref}
\allowdisplaybreaks

\allowdisplaybreaks
\newtheorem{theorem}{Theorem}[section]

\newtheorem{lemma}{Lemma}[section]

\newtheorem{remark}{Remark}[section]

\numberwithin{equation}{section}
 
\begin{document}
 \title{Infinitely many non-radial sign-changing solutions for   a Fractional Laplacian equation with critical nonlinearity}
\author{Fei Fang\footnote{Corresponding author. E-mail address: fangfei68@163.com\ \  Tel: +8613366101808}
\\School of Mathematical Sciences, Peking  University \\Beijing, 100871,  China}
 \date{ }
 \maketitle

\noindent \textbf{\textbf{Abstract:}} In this work, the following fractional Laplacian problem  with  pure  critical nonlinearity is considered
\begin{equation*}
\left\{
\begin{array}{ll}
(-\Delta)^{s} u=|u|^{\frac{4s}{N-2s}}u, &\mbox{in}\    \mathbb{R}^N, \\
u\in \mathcal{D}^{s,2}(\mathbb{R}^N),
\end{array}
\right.
\end{equation*}
where $s\in (0,1)$,  $N$ is a  positive integer with $N\geq 3$, $(-\Delta)^{s}$ is the fractional Laplacian  operator. We will prove that
this problem has infinitely many non-radial sign-changing solutions.

%\tableofcontents
\section{Introduction}
We tackle the following Fractional Laplacian problem with pure critical nonlinearity
\begin{equation}
\left\{
\begin{array}{ll}
(-\Delta)^{s} u=|u|^{p-1}u, &\mbox{in}\    \mathbb{R}^N, \\
u\in D^{s,2} (\mathbb{R}^N),
\end{array}
\right.    \label{P2}
\end{equation}
where $s\in (0,1)$, $ p=2^{\ast}-1$, $2^{\ast}=\frac{2N}{N-2s}$, $N$ is a  positive integer with $N\geq 3$, and the space $D^{s,2}(\mathbb{R}^N)$  is
defined as the completion of $C_0^{\infty}(\mathbb{R}^N)$ under the norm:
$$\|u\|_{\mathcal{D}^{s,2}(\mathbb{R}^N)}:=\int_{\mathbb{R}^N}|\xi|^{2s}\hat{u}d\xi=\int_{\mathbb{R}^N}|(-\Delta)^{\frac{s}{2}}u(x)|^2dx,$$
where $\ \hat{ }\ $  is Fourier transform. $(-\Delta)^{s}$ is a Fractional Laplacian operator  which  is  defined as a pseudo-differential operator:
$$\widehat{(-\Delta)^{s}} u(\xi))=|\xi|^{2s}\hat{u}(\xi).$$
If $u$ is smooth enough, it can also be computed by the following singular integral:
$$(-\Delta)^{s} u(x)=C_{N,s} P.V. \int_{\mathbb{R}^N}\frac{u(x)-u(y)}{|x-y|^{N+2s}}dy.$$
Here P.V. is the principal value and $C_{N,s}$ is a normalization constant.  The operator $(-\Delta)^{s}$
can be seen as the infinitesimal generators of L\'{e}vy stable diffusion processes (see \cite{r9}). This
operator arises in several areas such as physics, biology, chemistry and finance (see \cite{r9,r10}). In
recent years, The fractional Laplacians have recently attracted much research interest, and there are
a lot of results in the literature on the existence of
such solutions, e.g.,  ( \cite{r12,r13,r14,r15,r16,r17,r11, r4, r7,r8}) and the references therein.
In the remarkable work of Caffarelli and
Silvestre \cite{r15}, this nonlocal operator   can be defined by
the following Dirichlet-to-Neumann map:
$$ (-\Delta)^{s} u(x) =-\frac{1}{k_s}\lim_{y\rightarrow0^{+}}\frac{\partial w}{\partial y}(x,y),$$
where, $k_s= (2^{1-2s}\Gamma(1-s)/\Gamma(s))$ and $w$  solves the boundary value problem
\begin{equation*}
  \left\{\begin{array}{ll}
           -\mbox{div}(y^{1-2s}\nabla w)=0 & \ \mbox{in}\ \mathbb{R}_{+}^{N+1}, \\
           w(x, 0)=u & \ \mbox{on}\ \mathbb{R}^{N}.
         \end{array}
  \right.
\end{equation*}
Caffarelli and his coauthors  \cite{r13, r14} investigated free boundary
problems of a fractional Laplacian. The  operator was studied  by  Chang and Gonz\'{a}lez \cite{r17} in conformal
geometry. Silvestre \cite{r22} obtained some regularity results for the obstacle problem of the
fractional Laplacian. Recently, Fractional  Schr\"{o}dinger equations with respect to wave standing waves were studied in
 \cite{r11,  r24, r25, r26,r19,r27}.
  Very recently,  the singularly perturbed
problem of fractional Laplacian was considered  by D\'{a}vila, del Pino and Wei \cite{r23},
and  they recovered various existence results already known for the case $s=1$.

Problem (\ref{P2}) arises from looking for a solution for the following fractional Nirenberg problem,
\begin{equation}\label{P1}
P_{s} (u)=|u|^{\frac{4s}{N-2s}}u,\ \ \mbox{in}\    \mathbb{S}^N,
\end{equation}
where $(\mathbb{S}^N, g_{\mathbb{S}^N} )$,\ $n\geq 2$, be the standard sphere in $\mathbb{\mathbb{R}}^{N+1}$,
$P_s$ is an intertwining operator. The reader is referred to \cite{r7, r8} for more details on fractional Nirenberg problem. Similar to
 the case of $s=1$, using the stereo-graphic projection, problem (\ref{P1}) can be reduced to  problem (\ref{P2}).

%In this paper, following the idea in \cite{r2, r3, r18}, we will construct a sequence of non-radial sign-changing
%solutions for problem \eqref{P2}.

This idea of this paper is motivated by the recent papers \cite{r2, r18}, where infinitely many solutions to the Yamabe
problem  and the Yamabe probelm of  polyharmonic  operator were constructed, respectively.

Our main result in this paper can be stated as follows:

\begin{theorem}\label{t2}
 Assume that $N\geq 3$, then problem \eqref{P2} has infinitely many non-radial sign-changing solutions.
\end{theorem}
We will prove Theorem \ref{t2} by proving the following result:
\begin{theorem}\label{t1}
Let $N\geq 3$ and write $\mathbb{R}^N=\mathbb{C}\times \mathbb{R}^{N-2}$ and $\xi_j=\sqrt{1-\mu^2}(e^{\frac{2\pi(j-1)}{k}i}, 0, \dots, 0), j=1,\cdots, k.$
Then for any sufficiently large $k$ there is finite energy solution of the form
$$u_k(x)=U(x)-\sum_{j=1}^k\mu_k^{-\frac{n-2s}{2}}U(\mu_k^{-1}(x-\xi_j))+o(1),$$
where $$\mu_k=\delta_k^{\frac{2}{N-2s}}k^{-3}, U_{x,\Lambda}(y)= c_{N,s}\left(\frac{\Lambda}{1+\Lambda^2|y-x|^2}\right)^{\frac{N-2s}{2}}, c_{N, s}>0, \Lambda>0, x\in \mathbb{R}^N$$
and $o(1)\rightarrow 0$ uniformly as $k\rightarrow \infty$, $\delta_k$ is a positive number which only depends  on $k$.
\end{theorem}

\begin{remark}\label{remark1}
We believe  that the similar result should also hold for the following  critical problems with the presence of weight:
\begin{equation*}
\left\{
\begin{array}{ll}
(-\Delta)^{s} u=K(x)|u|^{\frac{N+2s}{N-2s}}, &\mbox{in}\    \mathbb{R}^N, \\
u> 0,  &\mbox{on}\    \mathbb{R}^N ,\\
u\in D^{s,2}(\mathbb{R}^N).
\end{array}
\right.
\end{equation*}
Following the idea in \cite{r3}, for appropriate  weight $K(x)$, we can also construct a sequence of non-radial positive solutions
solutions for this  problem.  When $K(x)=1$,   it was shown independently by Y. Y. Li and Chen, Li, and Ou \cite{r21, r20} that for
$s\in(0,N/2)$ and $ u \in  L_{loc}^{\frac{2N}{N-2s}}(\mathbb{R}^N) $ the equation has a unique positive solution $u(r) > 0$ (see \eqref{e00})up to scaling and translation.
\end{remark}

\begin{remark}\label{remark2}
In a recent work, W. Long, S.J. Peng and J.  Yang \cite{r19} obtained infinitely many positive  solutions  for the following subcritical equation:
\begin{equation*}
(-\Delta)^{s} u+u=k(|x|)u^{q}, \ \mbox{in}\    \mathbb{R}^N,
\end{equation*}
where $q<\frac{N+2s}{N-2s}$, $K(|x|)$ is a positive radial function, and satisfies  some asymptotic assumptions  at infinity.
\end{remark}

We organize this paper as follows. In Section 2,  we construct an approximation solution and give the estimates of the error.  Section 3  contains a
linear result.  Section 4  will devote to the detailed calculus and further thoughts on the gluing procedures. The proof the main result will be given in  last section.

\section{Approximation solution and the estimate of the error}
We start with the construction of a first approximate solution to problem  (\ref{P2}). Then we give
the precise estimate of the error.
It is well known that the radial  functions
\begin{equation}\label{e00}
U_{x,\Lambda}(y)= c_{n,s}\left(\frac{\Lambda}{1+\Lambda^2|y-x|^2}\right)^{\frac{N-2s}{2}}, c_{n, s}>0, \Lambda>0, x\in \mathbb{R}^N,
\end{equation}
are the only solutions to the problem
\begin{equation}\label{e001}
-\Delta u= u^{\frac{N-2s}{N+2s}}, u>0,\ \mbox{in}\ \mathbb{R}^N.
\end{equation}
Moreover,  the radial solution $U$ is invariant under the Kelvin type transform
\begin{equation}\label{e1}
\hat{u}(y)=|y|^{2s-N}u \left(\frac{y}{|y|^2}\right).
\end{equation}
That is, $\hat{U}=U(y)$. Problem (\ref{e001}) is invariant under the Kelvin transform (\ref{e1}) (see \cite{r1}).

Let $$w_{\mu}(y-\xi)=\mu^{-\frac{N-2s}{2}}U\left(\mu^{-1}(y-\xi)\right).$$
Then a simple algebra computation implies that:
\begin{lemma}\label{l21}
$w_{\mu}(y-\xi)$ is invariant under the Kelvin type transform (\ref{e1}) if and only if $\mu^2+|\xi|^2=1$.
\end{lemma}

Let $k$  be a large positive integer and $\mu>0$ be a small concentration parameter such that:
\begin{equation}
\mu=\delta^{\frac{2}{N-2s}}k^{-3}
\end{equation}
where $\delta$ is a positive parameter that will be fixed later. Let
$$\xi_j=\sqrt{1-\mu^2}(e^{\frac{2\pi(j-1)}{k}i}, 0, \dots, 0), j=1,\cdots, k.$$
We denote $U_j(y):=w_{\mu}(y-\xi_j), j=1,\cdots, k,$
and consider the function
$$U_{\ast}(y):=U(y)-\sum_{j=1}^{k}U_j(y).$$

In order to obtain  sign-changing solutions for problem (\ref{P2}), we follow the method of \cite{r2,r18}
and use the number of the bubble solutions $U_j$ as a parameter.
The idea of using the number of bubbles as a parameter was first used by Wei
and Yan \cite{r3} in constructing infinitely many positive solutions to the prescribing scalar
curvature problem.   We will prove that when the
bubbles number $k$ is large enough,  problem (\ref{P2}) has a solution of the form:
$$u(y)=U_{\ast}(y)+\phi(y),$$
where $\phi$ is a small function  when compared with $U$. If  $u$  satisfies the above  form, then  problem (\ref{P2}) can be rewritten as
\begin{equation}\label{e3}
(-\Delta)^{s}\phi-p|U_{\ast}|^{p-1}\phi+E-N(\phi)=0,
\end{equation}
where $p=2^{\ast}-1, 2^{\ast}=\frac{2N}{N-2s}$, and
$$E=(-\Delta)^{s}U_{\ast}-|U_{\ast}|^{p-1}U_{\ast},$$
$$N(\phi)=|U_{\ast}+\phi|^{p-1}(U_{\ast}+\phi)-|U_{\ast}|^{p-1}U_{\ast}-p|U_{\ast}^{p-1}|\phi.$$

We will show that for sufficiently large $k$, the error term $E$ will be controlled small enough so that some asymptotic estimate holds.
In order to obtain the better control on the error, for a fixed number $\frac{N}{s}>q>\frac{N}{2s}$, we consider  the following weighted $L^q$ norm:
\begin{equation}\label{e23}
  \|h\|_{\ast\ast}:=\left\|(1+y)^{N+2s-\frac{2N}{q}}h(y)\right\|_{L^{q}(\mathbb{R}^{N})}
\end{equation}
and
\begin{equation}\label{e24}
\|\phi\|_{\ast}:=\left\|(1+y)^{N-2s}\phi(y)\right\|_{L^{\infty}(\mathbb{R}^{N})}.
\end{equation}

\begin{lemma}\label{l22}
 There exist an integer $k_0$ and a positive constant $C$ such that for all $k>k_0$, the following estimate  for $E$ is true:

 \begin{equation}\label{e4}
\|E\|_{\ast\ast}\leq
Ck^{-\frac{N}{q}-2s}.
\end{equation}
\end{lemma}
\begin{proof}
We prove this lemma in two steps.  Firstly, we consider the error term $E$ in the exterior region:
$$EX:=\bigcap_{j=1}^{k}B_{\xi_{j}}^c(\eta/k):=\bigcap_{j=1}^{k}\left\{|y-\xi_j|>\eta/k\right\}.$$
Here $\eta > 0$ is a positive and small constant, independent of $k$. Secondly, we will do it in the interior region:
$$IN=EX^{c}=\bigcup_{j=1}^{k}\left\{|y-\xi_j|\leq\eta/k\right\},$$
where $\eta\ll1$.

\noindent \textbf{Step 1}. By the mean value theorem, we have
\begin{align}\label{e25}
E&= (-\Delta)^{s}U_{\ast}-|U_{\ast}|^{p-1}U_{\ast}\nonumber\\
&=(-\Delta)^{s}\left[U-\sum_{j=1}^{k}U_j\right]-\left|U-\sum_{j=1}^kU_j\right|^{p-1}\left(U-\sum_{j=1}^{k}U_j\right)\nonumber\\
&=U^p-\sum_{j=1}^{k}U_j^{p}-\left|U-\sum_{j=1}^kU_j\right|^{p-1}\left(U-\sum_{j=1}^{k}U_j\right)\nonumber\\
&=-\left[p\left|U-t\sum_{j=1}^kU_j \right|^{p-1}\left(-\sum_{j=1}^{k}U_j\right)+\sum_{j=1}^{k}U_j^{p}\right]\nonumber\\
&=p\left|U-t\sum_{j=1}^kU_j \right|^{p-1}\left(\sum_{j=1}^{k}U_j\right)-\sum_{j=1}^{k}U_j^{p}, \ \mbox{for}\ t\in(0,1).
\end{align}
Now the exterior region  is divided into two parts, that is,

$$A_1:=\{y:|y|\geq 2\}\ \mbox{and}\ A_2:=\left\{|y|\leq 2\right\}\cap \left[\bigcap_{j=1}^{k}\left\{|y-\xi_j|>\eta/k\right\}\right].$$
For $y\in A_1$, one has $\frac{1}{|y-\xi|}\sim \frac{1}{1+|y|}$. So,
\begin{align}\label{e6}
  |E(y)|\leq & C\left\{(1+|y|^2)^{-2s}+\left[\sum_{j=1}^k\mu^{\frac{N-2s}{2}}(\mu^2+|y-\xi_j|^2)^{-\frac{N-2s}{2}}\right]^{\frac{4s}{N-2s}}\right\} \nonumber \\
  & \ \ \cdot\left[\sum_{j=1}^k\mu^{\frac{N-2s}{2}}(\mu^2+|y-\xi_j|^2)^{-\frac{N-2s}{2}}\right]\nonumber \\
  &\leq C\left[(1+|y|^2)^{-2s}+\frac{\mu^{2s }k^{\frac{4s}{N-2s}}}{(1+|y|^2)^{2s}}\right]\cdot \sum_{j=1}^{k}\frac{\mu^{\frac{N-2s}{2}}}{|y-\xi_j|^{N-2s}}\nonumber \\
  &\leq C\frac{\mu^{\frac{N-2s}{2}}}{(1+|y|^2)^{2s}} \sum_{j=1}^{k}\frac{1}{|y-\xi_j|^{N-2s}}.
\end{align}
For $y\in A_2$, let us consider two cases:

\noindent \textbf{(1)} There is a $i_0\in \{1,2,3,\cdots, k\}$ such that $y$ is closest to $\xi_{i_0}$, and far from  all the other $\xi_j$'s ($j\not=i_0$). Then,
$$|y-\xi_j|\geq \frac{1}{2}|\xi_{i_0}-\xi_j|\sim \frac{|j-i_0|}{k}.$$

\noindent \textbf{(2)} y is far from all $\xi_i$'s, that is, there is a $C_0>0$ such that $|y-\xi_i|\geq C_0(1\leq i\leq k).$
\begin{align}\label{e5}
  |E(y)|\leq & C\left\{(1+|y|^2)^{-2s}+\left[\sum_{j=1}^k\mu^{\frac{N-2s}{2}}(\mu^2+|y-\xi_j|^2)^{-\frac{N-2s}{2}}\right]^{\frac{4s}{N-2s}}\right\}\nonumber \\
  & \ \ \cdot\left[\sum_{j=1}^k\mu^{\frac{N-2s}{2}}(\mu^2+|y-\xi_j|^2)^{-\frac{N-2s}{2}}\right]\nonumber \\
  &\leq C\left[(1+|y|^2)^{-2s}+\left(\frac{\mu^{\frac{N-2s}{2}}}{|y-\xi_{i_0}|^{N-2s}}+\sum_{j\not=i_0 }\frac{\mu^{\frac{N-2s}{2}}}{|y-\xi_{j}|^{N-2s}}\right)^{\frac{4s}{N-2s}}\right]\cdot \sum_{j=1}^{k}\frac{\mu^{\frac{N-2s}{2}}}{|y-\xi_j|^{N-2s}}\nonumber \\
  &\leq C\left[(1+|y|^2)^{-2s}+\left(\mu^{2s}k^{4s}+\max\left\{\sum_{j\not=i_0 }\frac{\mu^{2s}k^{4s}}{|j-i_0|^{4}}, k^{\frac{4s}{N-2s}\mu^{2s}}\right\}\right)\right]\cdot \sum_{j=1}^{k}\frac{\mu^{\frac{N-2s}{2}}}{|y-\xi_j|^{N-2s}}\nonumber \\
  &\leq C\frac{\mu^{\frac{N-2s}{2}}}{(1+|y|^2)^{2s}} \sum_{j=1}^{k}\frac{1}{|y-\xi|^{N-2s}}.
\end{align}
Consequently, from \eqref{e6} and \eqref{e5}, in the exterior region, we obtain
\begin{align*}
  \|E\|_{\ast\ast} =&\|(1+|y|^{(N+2s)q-2N})E^q(y)\|_{L^q(EX)}  \\
\leq & C\mu^{\frac{N-2s}{2}}\sum_{j=1}^{k}\left[\int_{B_{\xi_j^c}(\eta/k)}\frac{(1+|y|)^{(N+2s)q-2N}}{(1+|y|)^{4s q}|y-\xi_j|^{(N-2s)q}}\right]^{1/q}\\
\leq& C\mu^{\frac{N-2s}{2}}k\left[\int_{\eta/k}^1\frac{r^{N-1}dr}{r^{(N-2s)q}}+\int_1^{+\infty}r^{-(N+1)}dr\right]^{1/q}\\
\leq& C(k^{1+s-\frac{N}{2}-\frac{N}{q}}+k^{1+3s-\frac{3N}{2}})\leq C k^{1+s-\frac{N}{2}-\frac{N}{q}}.
\end{align*}

\noindent \textbf{Step 2}. In the case of the interior region $IN$, we easily know that for all $y\in IN$, there is $j\in \{1,2,3,\cdots, k\}$, such that
$|y-\xi_j|\leq \eta/k$.  Similar to \eqref{e25}, we can obtain
\begin{equation}\label{e28}
E=p\left[U_j-t\left(-\sum_{j\not=i}^kU_i+U \right)\right]^{p-1}\cdot\left(-\sum_{j\not=i}^kU_i+U \right)+\left(\sum_{j\not=i}^kU_i \right)^{p}-U^p, \ \mbox{for}\ t\in(0,1).
\end{equation}
In order to measure the error term $E$, we define
$$\tilde{E}_j=\mu^{\frac{N+2s}{2}}E(\xi_j+\mu y).$$
We observe that
$$\mu^{\frac{N-2s}{2}}U_j(\xi_j+\mu y)=U(y)\ \mbox{and}\  U_i(y)=\mu^{-\frac{N-2s}{2}}U(\mu^{-1}(y-\xi_i))$$
Thus, for $i\not=j$,
\begin{equation}\label{e29}
\mu^{\frac{N-2s}{2}}U_i(\xi_j+\mu y)=U(y-\mu^{-1}(\xi_i-\xi_j)).
\end{equation}
Note also that
$\mu^{-1}|\xi_i-\xi_j|\sim \frac{|i-j|}{k\mu}$.
Hence, by \eqref{e28} and  \eqref{e29}, we estimate
\begin{align*}
  |\tilde{E_j}(y)|\leq & C\left|U(y)+\sum_{i\not= j}\frac{(k\mu)^{N-2s}}{|j-i|^{N-2s}}+\mu^{\frac{N-2s}{2}}U(\xi_j+\mu y)\right|^{p-1} \\
  & \ \ \cdot\left( \sum_{i\not= j}\frac{(k\mu)^{N-2s}}{|j-i|^{N-2s}}+\mu^{\frac{N-2s}{2}}U(\xi_j+\mu y)\right)+\sum_{i\not= j}\left(\frac{(k\mu)^{N-2s}}{|j-i|^{N-2s}}\right)^p+\mu^{\frac{N+2s}{2}}U^p(\xi_j+\mu y)\\
  &\leq C\left|\mu^{\frac{N-2s}{2}}+\left(\frac{1}{1+|y|^2}\right)^{\frac{N-2s}{2}}\right|^{p-1}\cdot
  \mu^{\frac{N-2s}{2}}+\mu^{\frac{N-2s}{2}p}+\mu^{\frac{N+2s}{2}}\\
  &\leq C\left|\mu^{\frac{N+2s}{2}}+\frac{\mu^{\frac{N-2s}{2}}}{1+|y|^{4s}}\right|\\
  &\leq C\frac{\mu^{\frac{N-2s}{2}}}{1+|y|^{4s}}.
\end{align*}
Therefore, we can get
\begin{align*}
\|E\|_{\ast\ast(|x-\xi_j|<\eta/k)}&=\left(\int_{|x-\xi_j|<\eta/k} (1+|x|)^{(N+2sq)-2N} \mu^{-\frac{N+2s}{2}q} \tilde{E}^q\left(\frac{x-\xi_j}{\mu}\right)dx\right)^{\frac{1}{q}}\\
&\leq C\left[\int_{|y|\leq \eta/(k\mu)}\left|\mu^{\frac{N}{q}-\frac{N+2s}{2}}\tilde{E}_j(y)\right|^qdy\right]^{1/q}\\
&\leq C\left[\mu^{N-2qs}\int_0^{\eta/(k\mu)}\frac{r^{N-1}}{1+r^{4qs }}dr\right]^{1/q}\\
&\leq C\mu^{2s}k^{-\frac{N}{q}+4s}\leq Ck^{-\frac{N}{q}-2s}.
\end{align*}

Finally, by combining the estimates in the exterior region and interior region together, we have
\begin{align*}
\|E\|_{\ast\ast}&\leq \|E\|_{\ast\ast(EX)} +\|E\|_{\ast\ast(IN)}\\
&\leq \|E\|_{\ast\ast(EX)} +\sum_{j=1}^k\|E\|_{\ast\ast(IN)(|x-\xi_j|<\eta/k)}\\
&\leq Ck^{-\frac{N}{q}-2s}.
\end{align*}
\end{proof}

\section{A linear result}

We consider the operator $L_0$ defined as
$$L_0(\phi):=(-\Delta)^{s} \phi-pU^{p-1}\phi,\  \mbox{with}\ p=m^{\ast}-1.$$
We can know that (see \cite{r4}) the solution space for the homogeneous equation $L_0(\phi)=0$ is spanned by the $n+1$ functions,
$$v_i=\partial_{y_i}U,\  i=1,2, \cdots, N; \ v_{N+1}=x\cdot\nabla U+\frac{n-2s}{2}U.$$
This section is devoted to establishing an invertibility theory for
\begin{equation}\label{e31}
L_0(\phi)=h\ \mbox{in}\ \mathbb{R}^N.
\end{equation}

It  is worth noting that  the $L^p$  to $W^{2s,p}$ estimate does not hold for all $p$ in this fractional framework (see Remarks 7.1
and 7.2 in \cite{r5}).  But we have  he following Lemma \ref{l30}. This is an important  ingredient in the proof of Lemma \ref{l31}.
\begin{lemma}[\cite{r5,r6}]\label{l30}
Let $\Omega\subset \mathbb{R}^N$ be a bounded $C^{1,1}$ domains, $s\in (0,1)$, $n>2s$, $g\in C(\bar{\Omega})$, and $u$ be the solution of

\begin{equation}
\left\{
\begin{array}{ll}
(-\Delta)^{s} u=g, &\mbox{in}\    \Omega, \\
u> 0,  &\mbox{on}\    \mathbb{R}^N\setminus \Omega.
\end{array}
\right.    \label{P4}
\end{equation}
\begin{compactitem}
  \item[\emph{(i)}] For each $1\leq r<\frac{n}{n-2s}$, there exists a constants $C$ depending only $n, s, r$ and $|\Omega|$, such that
  $$\|u\|_{L^r(\Omega)}\leq C\|g\|_{L^1(\Omega)},\ r<\frac{n}{n-2s}.$$
  \item[\emph{(ii)}] Let  $1\leq p<\frac{n}{2s}$. Then there exists a constant $C$ depending only on $n,s$ and $p$ such that
   $$\|u\|_{L^q(\Omega)}\leq C\|g\|_{L^p(\Omega)},\ \mbox{where}\  q=\frac{np}{n-2ps}.$$
  \item[\emph{(iii)}]  Let  $\frac{n}{2s}\leq p<\infty$. Then, there exists a constants $C$ depending only on  $n,s, p$ and $|\Omega|$  such that
   $$\|u\|_{C^{\beta}(\Omega)}\leq C\|g\|_{L^p(\Omega)},\ \mbox{where}\  \beta=\min\left\{s, 2s-\frac{n}{p}\right\}.$$
\end{compactitem}

\end{lemma}

\begin{lemma}\label{l31}
Let $h(y)$ be a function such that $\|h\|_{\ast \ast}<\infty$. Assume that $\frac{N}{2s}<q<\frac{N}{s}$,  and
$$ \int_{\mathbb{R}^N}v_l h=0, \forall l=1,2,\cdots, N+1.$$
Then problem \eqref{e31} has a unique solution satisfying $\|\phi\|_{\ast}<\infty$ and
$$ \int_{\mathbb{R}^N}U^{p-1}v_l h=0, \forall l=1,2,\cdots, N+1.$$
Furthermore, there exists  a constant $C$ which depends on $q$ and $N$ such that
$$\|\phi\|_{\ast}\leq C\|h\|_{\ast\ast}.$$
\end{lemma}
\begin{proof}
Set
$$H=\left\{ \phi\in \mathcal{D}^{s,2}(\mathbb{R}^N): \int_{\mathbb{R}^N} U^{p-1}v_l \phi dx=0, \forall l=1,2,\cdots, N+1 \right\}.$$
Then $H$ is a Hilbert space equipped  with the following  norm:
$$\| u\|_{H}:=\int_{\mathbb{R}^N} |\xi|^{2s}|\hat{u}(\xi)|^2d\xi=\int_{\mathbb{R}^N}|(-\Delta)^{\frac{s}{2}}u(y)|^2dy.$$
Furthermore, for all $\psi\in H$,  it is easy to show that
$$(L_0\phi, \psi)_{H}=(\phi, L_0\psi)_{H}.$$

Let $r=\frac{2N}{N+2s}, r'=\frac{2N}{N-2s}=p+1$. Since $\|h\|_{\ast} \leq \infty$, the H\"{o}lder inequality implies
\begin{align}\label{e32}
\|h\|_r&\leq  \left[\int_{\mathbb{R}^N} \left|h^r (1+|y|)^{(N+2s)r-\frac{2Nr}{q}}\cdot(1+|y|)^{-(N+2s)r+\frac{2Nr}{q}}\right|dy\right]^{\frac{1}{r}} \nonumber\\
&\leq \left[\int_{\mathbb{R}^N}|h|^q(1+|y|)^{(N+2s)q-2N}dy\right]^{1/q}\left[\int_{\mathbb{R}^N}(1+|y|)^{-2N}\right]^{\frac{1}{r}-\frac{1}{q}}\nonumber\\
&\leq C\|h\|_{\ast\ast}<\infty,
\end{align}
and
\begin{align}\label{e33}
\|U^{p-1}\phi\|_r&\leq \left(\int_{\mathbb{R}^N}|\phi|^{r\cdot\frac{N+2s}{N-2s}}\right)^{\frac{N-2s}{(N+2s)r}}\cdot
\left(\int_{\mathbb{R}^N}|U|^{(p-1)r\cdot\frac{N+2s}{4s}}\right)^{\frac{4s}{(N+2s)r}}\nonumber\\
&=\|\phi\|_{p+1}\cdot \left(\int_{\mathbb{R}^N}|U|^{\frac{2N}{N-2s}}\right)^{\frac{2s}{N}}\nonumber\\
&\leq C\|\phi\|_{p+1}\leq C\|\phi\|_H<\infty \ (\mbox{by\  fractional \ Sobolev\  inequality}).
\end{align}
For  $f\in L^r$, using \eqref{e32} and  \eqref{e33}, the weak solution can be considered  by the following equation

\begin{align}\label{e34}
\int_{\mathbb{R}^N}(-\Delta)^{\frac{s}{2}}\phi \cdot (-\Delta)^{\frac{s}{2}}\psi+\int_{\mathbb{R}^N} f\psi=0,\mbox{for\  all} \ \psi\in H.
\end{align}
Define the functional $A_f: H \rightarrow \mathbb{R}$ as follow:
$$A_f(\psi)=\int_{\mathbb{R}^N}f\psi,$$
and we easily know that
\begin{align}\label{e35}
\int_{\mathbb{R}^N}(-\Delta)^{\frac{s}{2}}\phi \cdot (-\Delta)^{\frac{s}{2}}\psi= A_f(\psi).
\end{align}
Furthermore, using H\"{o}lder inequality  again, one has
$$|A_f(\psi)|\leq \|f\|_r\|\psi\|_{p+1}\leq C\|f\|_r\|\psi\|_{H},$$
so this shows that $A_{f}$ is a bounded linear functional on the Hilbert space $H$. Applying the Riesz representation theorem, there is
a unique $\phi\in H$ such that
$$A_f(\psi)=\int_{\mathbb{R}^N}(-\Delta)^{\frac{s}{2}}\phi \cdot (-\Delta)^{\frac{s}{2}}\psi.$$
Hence, according to the functional $A_f$, one can define a new operator $A: L^r\rightarrow H$ by
$$A(f)=\phi\ \mbox{and}\ \langle A(f), \psi\rangle_H=\langle f, \psi\rangle, \forall \psi \in H.$$
Then (\ref{e31}) can be formulated as
$$ \phi= A(h)+A(pU^{p-1}\phi)=A(h)+A(\tau (\phi)),$$
where $\tau: H\rightarrow L^r, \phi\rightarrow pU^{p-1}\phi,$ is a compact mapping, thanks to local compactness of Sobolev's embedding and the fact that $U^{p-1}=O(|y|^{-4})$.

Let $B=A\circ \tau$, then we easily show that $B$ is the  operator from $H$ to $H$, and also compact,  self-adjoint. Hence problem \eqref{e31} can be equivalent to
$$(I-B)\phi=A(h).$$
Now   Fredholm alternative theorem  tell us that  the above equation has a solution if and only if
$$\forall v\in ker(I-B),\ \ (I-B)v=0=A(0).$$
Therefore, we get $h\equiv 0$ and
$$A(0)\in R(I-B)=(Ker(I-B^{\ast}))^{\perp}=(Ker(I-B))^{\perp}.$$
So problem \eqref{e31} be simplified as the homogeneous formula, namely,
$$L_0(v)=0,$$
where $v$ can be denoted by the sum of $v_i$'s, that is,
$$v(y)=\sum_{j=1}^{N+1}a_jv_j(y),$$
with constants $a_1, a_2,\cdots, a_{N+1}$.

By the definition of $H$, we have
$$0=\int_{\mathbb{R}^N}U^{p-1}v_j\cdot v=a_j\int_{\mathbb{R}^N} U^{p-1}v_j,$$
which yields the vanishing terms
$$a_j=0, j=1, 2,\cdots, N+1, \ \mbox{and}\ v\equiv 0, ker(I-B)=\{0\}.$$
Consequently, the orthogonal terms $R(I-B)=H$; this implies the existence of $\phi$ by
$$(I-B)\phi=A(h),$$
and the uniqueness of $\phi$ by
$$ker(I-B)=\{0\}.$$

In the following,  we will prove that $$\|\phi\|_{\ast}\leq C\|h\|_{\ast\ast}.$$  By Lemma \ref{l30}(iii), we have
\begin{equation}\label{e365}
 \|\phi\|_{L^{\infty}(B)}\leq  C\|(1+|y|^{N+2s-\frac{2N}{q}})h\|_{L^q(\mathbb{R}^N)}.
\end{equation}
Now let us consider Kelvin's transform of $\phi$,
$$\tilde{\phi}(y)=|y|^{2s-N}\phi(|y|^{-2} y).$$
Then we easily see that $\tilde{\phi}$ satisfies the equation
\begin{equation}\label{e37}
  (-\Delta)^{s} \tilde{\phi}-pU^{p-1}\tilde{\phi}=\tilde{h},\  \mbox{in}\ \mathbb{R}^N,
\end{equation}
where $\tilde{h}(y)=|y|^{-2s-N}h(|y|^{-2}y)$.
Note that
\begin{equation}\label{e38}
\|\tilde{h}\|_{L^q(|y|<2)}=\||y|^{N+2s-\frac{2N}{q}}h\|_{L^q(|y|>1/2)}\leq C\|(1+|y|)^{N+2s-\frac{2N}{q}}h\|_{L^q(\mathbb{R}^N)}
\end{equation}
Applying  Lemma \ref{l30} to \eqref{e37},  and by \eqref{e38}   we obtain
\begin{equation}\label{e385}
\|\tilde{\phi}\|_{L^{\infty}(B)}\leq  C\|\tilde{h}|_{L^q(B_2)}\leq C\|(1+|y|)^{N+2s-\frac{2N}{q}}h\|_{L^q(\mathbb{R}^N)}.
\end{equation}
But
\begin{equation}\label{e39}
\|\tilde{\phi}\|_{L^{\infty}(B)}=\||y|^{n-2s}\phi\|_{L^{\infty}(\mathbb{R}^N\setminus B_1)}.
\end{equation}
Finally,    \eqref{e365},  \eqref{e385} and  \eqref{e39} imply that
$$\|\phi\|_{\ast}\leq C\|h\|_{\ast\ast}.$$

\end{proof}

\section{A gluing procedure}
Let $\zeta(s)$ ba a smooth function satisfying
\begin{equation*}
  \zeta=\left\{\begin{array}{ll}
                 1 &  0\leq s<\frac{1}{2}; \\
                 \mbox{smooth} & \frac{1}{2}\leq s\leq 1; \\
                 0 &  s>1.
               \end{array}
  \right.
\end{equation*}
The cut-off functions are defined by
\begin{equation*}
  \zeta_j=\left\{\begin{array}{ll}
                 \zeta(k\eta^{-1}|y|^{-2}|(y-|y|^2\xi_j)|) &  \mbox{if}\ |y|\geq 1|; \\
                 \zeta(k\eta^{-1}|(y-\xi_j)|) &   \mbox{if}\ |y|< 1|.
               \end{array}
  \right.
\end{equation*}
Then a simple  algebra computation implies that
$$\zeta_j(y)=\zeta_j(|y|^{-2}y),\ \ \mbox{supp}(\zeta_j)\subset \{y: |y-\xi_j|\leq \eta/k\}, j=1,2,\cdots, k.$$

Let $\phi=\sum_{k=1}^k\tilde{\phi}_j+\psi,$ $\bar{y}=(y_1, y_2), y'=(y_3,\cdots, y_N)$, and we assume
\begin{equation}\label{e41}
\tilde{\phi}_j(\bar{y}, y')=\tilde{\phi}_1(e^{-\frac{2\pi(j-1)}{k}\sqrt{-1}}\bar{y}, y'),\ j=1,\cdots,k,
\end{equation}

\begin{equation}\label{e42}
\tilde{\phi}_1(y)=|y|^{2s-N}\tilde{\phi}_1(|y|^{-2}y),
\end{equation}

\begin{equation}\label{e43}
\tilde{\phi}_1(y_1, \cdots,y_s, \cdots, y_N)=\tilde{\phi}_1(y_1, \cdots,-y_s, \cdots, y_N), s=2,3,\cdots,N,
\end{equation}
and
\begin{equation}\label{e44}
\|\phi_1\|_{\ast}\leq \rho \ \mbox{with}\ \rho\ll 1,
\end{equation}
where $\phi_1(y):=\mu^{\frac{N-2s}{2}}\tilde{\phi}_1(\xi_1+\mu y)$.
Now, due to the cut-off function,  problem \eqref{e3} can be split into the following system:

\begin{equation}\label{e45}
 \left\{\begin{array}{l}
                  (-\Delta)^s\tilde{\phi}_j-p|U_{\ast}|^{p-1}\zeta_j\tilde{\phi}_j+\zeta_j\left[-p|U_{\ast}|^{p-1}\psi+E-N\left(\tilde{\phi}_j+\displaystyle\sum_{i\not=j} \tilde{\phi}_i+\psi \right)\right]=0, \ j=1,\cdots,k, \\
                (-\Delta)^s\psi-pU^{p-1}\psi+\left[-p(|U_{\ast}|^{p-1}-U^{p-1})
                \left(1-\displaystyle\sum_{j=1}^k\zeta_j\right)+pU^{p-1}\left(\displaystyle\sum_{j=1}^k\zeta_j\right)\right]\psi\\
                \ \ \ -p|U_{\ast}|^{p-1}\displaystyle\sum_{j=1}^k\left(1-\zeta_j\right)\tilde{\phi}_j
                +\displaystyle\left(1-\sum_{j=1}^k\zeta_j\right)\left(E-N\left(\displaystyle\sum_{j=1}^k \tilde{\phi}_i+\psi \right)\right)=0.
               \end{array}
  \right.
\end{equation}

In order to obtain the existence and uniqueness of solution $\psi$ for the equation in \eqref{e45}, we can  simplify the  last  equation in \eqref{e45} to
\begin{equation}\label{e46}
 (-\Delta)^s\psi-pU^{p-1}\psi +(V_1+V_2)\cdot\psi+M(\psi)=0,
\end{equation}
where
\begin{equation}\label{e47}
  V_1=-p(|U_{\ast}|^{p-1}-U^{p-1})\left(1-\displaystyle \sum_{j=1}^k\zeta_j \right),\ \ V_2=pU^{p-1}\left(\displaystyle \sum_{j=1}^k\zeta_j \right),
\end{equation}

\begin{equation}\label{e48}
  M(\psi)=-p|U_{\ast}|^{p-1}\displaystyle\sum_{j=1}^k\left(1-\zeta_j\right)\tilde{\phi}_j
                +\displaystyle\left(1-\sum_{j=1}^k\zeta_j\right)\left(E-N\left(\displaystyle\sum_{j=1}^k \tilde{\phi}_i+\psi \right)\right),
\end{equation}
and
\begin{equation}\label{e49}
N(\phi)=|U_{\ast}+\phi|^{p-1}(U_{\ast}+\phi)-|U_{\ast}|^{p-1}U_{\ast}-p|U_{\ast}|^{p-1}\phi.
\end{equation}

\begin{lemma}\label{l41}
Assume that  $\tilde{\phi}_j$  satisfy the conditions \eqref{e41}-\eqref{e44}. Then there exist constants $k_0, \rho_o, C$  such that for all $k\geq k_0$ and $\rho<\rho_0$ problem \eqref{e46} has a unique solution $\psi=\Psi(\phi_1)$ which satisfies the following  symmetrical properties

\begin{equation}\label{e410}
\psi_1(\bar{y}, y_3,  \cdots,y_l, \cdots, y_N)=\psi(\bar{y}, y_3, \cdots,-y_l, \cdots, y_N);
\end{equation}

\begin{equation}\label{e411}
\psi(y)=|y|^{2s-N}\psi(|y|^{-2}y);
\end{equation}

\begin{equation}\label{e412}
\psi(\bar{y}, y')=\psi(e^{\frac{2\pi j}{k}\sqrt{-1}}\bar{y}, y'),\ j=1,\cdots,k-1.
\end{equation}
Furthermore
\begin{equation*}
           \|\psi\|_{\ast} \leq C\left[k^{1+s-\frac{N}{2}-\frac{N}{q}}+\|\phi_1\|_{\ast}^2\right]
\end{equation*}
and the operator $\Psi$ satisfies the Lipschitz property
$$\|\Psi(\phi_1^1)-\Psi(\phi_1^2)\|_{\ast}\leq C\|\phi_1^1-\phi_1^2\|_{\ast}.$$
\end{lemma}
\begin{proof}
Firstly,  consider linear problem    \eqref{e31}, and assume that $h$ satisfies the properties \eqref{e410}-\eqref{e412}, namely,
  \begin{equation}\label{e413}
h_1(\bar{y}, y_3,  \cdots,y_l, \cdots, y_N)=h(\bar{y}, y_3, \cdots,-y_l, \cdots, y_N);
\end{equation}

\begin{equation}\label{e414}
h(y)=|y|^{2s-N}h(|y|^{-2}y);
\end{equation}

\begin{equation}\label{e415}
h(\bar{y}, y')=h(e^{\frac{2\pi j}{k}\sqrt{-1}}\bar{y}, y'),\ j=1,\cdots,k-1.
\end{equation}

We will prove that
 \eqref{e31} has a unique bounded solution $\psi=T(h)$  and there exists a constant $C$ depending on $q$ and $N$ such that
 $$\|\psi\|_{\ast}\leq C\|h\|_{\ast\ast}$$
On  the basis of the results in Lemma \ref{l31}, we need to verify that
$$(h, v_l)=\int_{\mathbb{R}^N} hv_l=0,\ \forall l=1,2, \cdots, N+1.$$
  By the assumption \eqref{e413} that  $h$ is an even function, and the oddness of $v_l=\frac{\partial U}{\partial y_1}$, we easily know that $(h, v_l)=0$
  for $l=3,\cdots, N.$

  For the cases $l=1,2$, we consider the vector integral
  $$I=\int_{\mathbb{R}^N} h\left[\begin{array}{l}
                                   v_1 \\
                                   v_2
                                 \end{array}
  \right]dy=c_N\int_{\mathbb{R}^N} \frac{h(y)}{(1+|y|^2)^{\frac{N}{2}-1+s}}\cdot \left[\begin{array}{c}
                                   y_1 \\
                                   y_2
                                 \end{array}
  \right] dy.$$
 Set
 $$(\tilde{z}, z')=(e^{\frac{2\pi j}{k}\sqrt{-1}}\bar{y}, y').$$
  From the assumption \eqref{e414}, it is easy to see that
\begin{align}\label{e416}
 e^{\frac{2\pi j}{k}\sqrt{-1}}I&=c_N \int_{\mathbb{R}^N} \frac{h(y)}{(1+|y|^2)^{\frac{N}{2}-1+s}}\left[\begin{array}{l}
                                   y_1 \\
                                   y_2
                                 \end{array}
  \right]e^{\frac{2\pi j}{k}\sqrt{-1}} dy\\
  &=c_N\int_{\mathbb{R}^N} \frac{h(z)}{(1+|z|^2)^{\frac{N}{2}-1+s}}\cdot \left[\begin{array}{c}
                                   z_1 \\
                                   z_2
                                 \end{array}
  \right] dz\\
  &=I,
\end{align}
which implies $I=0$, since $e^{\frac{2\pi j}{k}\sqrt{-1}}\not =0$ for $k\geq 2$.

 For the case $l=N+1$, let us consider the following  function $I(\lambda)$, $\lambda>0$,
  $$I(\lambda)=\lambda^{\frac{N-2s}{2}}\int_{\mathbb{R}^N}U(\lambda y)h(y)dy.$$
  Using the transformation $z=|y|^{-2}y$, we obtain
  \begin{align}\label{e419}
   I(\lambda)& =\lambda^{\frac{N-2s}{2}}\int_{\mathbb{R}^N} U(\lambda y)h(y)dy \nonumber \\
    &= \lambda^{\frac{N-2s}{2}}\int_{\mathbb{R}^N} U(\lambda|y|^{-2} y)h(|y|^{-2}y)d(|y|^{-2}y)\nonumber \\
    & =\lambda^{-\frac{N-2s}{2}}\int_{\mathbb{R}^N} U(\lambda^{-1} y)h(y)dy\nonumber  \\
    &= I(\lambda^{-1}):=g(\lambda).
  \end{align}
  This implies
  $$ I'(1)=g'(1)=-\frac{1}{\lambda^2}I'\left(\frac{1}{\lambda}\right)\bigg|_{\lambda=1}=-I'(1).$$
  So $$0=I'(1)=(h, v_{N+1}).$$
  Therefore, by Lemma \ref{l31}, we have
   $$\|T(h)\|_{\ast}=\|\psi\|_{\ast}\leq C\|h\|_{\ast\ast}.$$

Next we will show that problem (\ref{e46}) has a unique solution. Taking $h=(V_1+V_2)\psi+M(\psi)$, then we write our problem
in fixed point form as
$$\mathcal{M}(\psi):=-T((V_1+V_2)+M(\psi))=\psi,\ \ \psi\in X,$$
where $X$ denotes the linear space with bounded norm $\|\cdot\|_{\ast}$ and satisfies all symmetry properties in Lemma \ref{l41}.
Now noting  that $$[(V_1+V_2\cdot\psi)+M(\psi)](y)=|y|^{-(N+2s)}[(V_1+V_2\cdot\psi)+M(\psi)]\left(|y|^{-2}y\right),$$
we can show that
 \begin{equation}\label{e420}
\psi_l(y)=\psi(\bar{y}, y_3,  \cdots,-y_l, \cdots, y_N), l=3,4,\cdots,N;
\end{equation}

\begin{equation}\label{e421}
\psi_{N+1}(y)=|y|^{2s-N}h(|y|^{-2}y);
\end{equation}
and
\begin{equation}\label{e422}
\psi^{j}(y)=\psi(e^{\frac{2\pi j}{k}\sqrt{-1}}\bar{y}, y'),\ j=1,\cdots,k-1
\end{equation}
satisfy problem \eqref{e31}.
Hence, by the unique result of Lemma \ref{e31}, we have
$$\psi=\psi_l=\psi^j=\psi_{N+1},$$
which are exactly the symmetries required in Lemma \ref{l41}.

The rest of this section is devoted to proving that $\mathcal{M}$ ia a contraction mapping. To do this, we must make
a series of estimate of $V_1, V_2, M$,  respectively.

Recall that $$V_1=-p(|U_{\ast}|^{p-1}-U^{p-1})\left(1- \sum_{j=1}^k\zeta_j \right);$$
and the assumptions of the cut-off functions $\zeta_j$ imply that $\mbox{supp}V_1\subset EX$

Now by taking the similar estimate as in the discussion of Step 1 in Lemma  \ref{l22}, for all $y\in EX$, there is a $s\in(0,1)$
such that
\begin{align}\label{e424}
  |V_1(y)\psi (y)|=& \left| V_1(y)\psi(y)(1+|y|^{N-2s})\frac{1}{1+|y|^{N-2s}}\right| \nonumber\\
  \leq  & C|V_1(y)U(y)|\cdot|(1+|y|^{N-2s})\psi| \nonumber\\
  \leq  & C|V_1(y)U(y)|\|\psi(y)\|_{\ast} \nonumber\\
  \leq  & C|U^{p-1}(y)-U_{\ast}^{p-1}(y)|\cdot\|\psi\|_{\ast} U(y)\nonumber\\
    \leq  & \left|U(y)-s \sum_{j=1}^kU_{j}(y)\right|^{p-2}\left[\sum_{j=1}^kU_{j}(y)\right]\cdot\|\psi\|_{\ast} U(y)\nonumber\\
\end{align}
Since $EX=A_1\cup A_2$, for $y\in A_1$, we get $ \mu^2+|y-\xi_j|^2\sim 1+|y|^2$. So
\begin{align}\label{e425}
  \sum_{j=1}^kU_{j}(y) &\leq  C k\mu ^{\frac{N-2s}{2}}U\leq   Ck^{3s+1-\frac{3N}{2}} U(y)\leq CU(y).
\end{align}
For $y\in A_2$, we obtain

\begin{align}\label{e426}
  \sum_{j=1}^kU_{j}(y)&=\sum_{j=1}^k\mu^{\frac{N-2s}{2}} \left(\frac{1}{\mu^2+|y-\xi_j|^2}\right)^{\frac{N-2s}{2}}\nonumber\\
  &\leq   Ck\cdot k^{\frac{N-2s}{2}}\mu^{\frac{N-2s}{2}}\leq Ck^{1+2s-N} \leq CU(2)\leq CU(y).
\end{align}
From \eqref{e425} and \eqref{e426}, we infer
\begin{equation}\label{e4265}
  \sum_{j=1}^kU_{j}(y)\leq CU(y), \mbox{for\ all\ } y\in EX.
\end{equation}
Using \eqref{e4265}, we can amplify \eqref{e424}, and obtain
\begin{align}\label{e427}
 |V_1(y)\psi(y)| &\leq C\|\psi\|_{\ast} U^{p-1}\left(\sum_{j=1}^kU_{j}(y)\right) \nonumber\\
   & \leq \|\psi\|_{\ast}\left(\frac{1}{(1+|y|^2)}\right)^{2s}\sum_{j=1}^k\frac{\mu^\frac{{N-2s}}{2}}{|y-\xi_j|^{N-2s}}
\end{align}
By  taking the similar estimate as in Step 1 in Lemma \ref{l22}, we see that
\begin{align*}
\|V_1\psi\|_{\ast\ast}\leq Ck^{1+s-\frac{N}{2}-\frac{N}{q}}.
\end{align*}
Now we turn  to  the estimate $V_2\psi$. Recall that
$$V_2=pU^{p-1}\left(\displaystyle \sum_{j=1}^k\zeta_j\right).$$
The  assumptions of the cut-off functions $\zeta_j$ imply that $\mbox{supp} V_2$ lies in the annular, that is,

 $$\mbox{supp} V_2\subset IN\subset \left\{y:||y|-\sqrt{1-\mu^2}|\leq \frac{\eta}{k}\right\}:=AN.$$
Therefore
\begin{align}\label{e428}
  \|V_2\psi\|_{\ast\ast} &\leq C\|\psi\|_{\ast}\sum_{j=1}^k\|U^p\zeta_j\|_{\ast\ast(|y-\xi_j|\leq \eta/k)} \nonumber\\
   & \leq \|\psi\|_{\ast}\cdot m(AN)\nonumber\\
     & \leq C \|\psi\|_{\ast}\cdot k^{1-N},
\end{align}
where $m(\cdot)$ denotes Lebesgue measure.

 We know that  $M$ is a nonlinear  operator, so the estimate of $M$ will be more complicated. For convenience sake, we introduce
 the following notations.

 Set

$$M_1=-p|U_{\ast}|^{p-1}\displaystyle\sum_{j=1}^k\left(1-\zeta_j\right)\tilde{\phi}_j,$$

$$M_2=\left(1-\sum_{j=1}^k\zeta_j\right)E,$$
and
$$
  M_3(\psi)=-\left(1-\sum_{j=1}^k\zeta_j\right)
                N\left(\displaystyle\sum_{j=1}^k \tilde{\phi}_i+\psi \right).
$$
Then the nonlinear operator can be denoted by
$$M(\psi)=M_1+M_2+M_3(\psi).$$
For $M_1$, applying the estimate of the exterior region and \eqref{e41}, \eqref{e42} and \eqref{e44}, we get
\begin{align}\label{e429}
  \|M_1\|_{\ast} &\leq C\sum_{j=1}^k \||U_{\ast}|^{p-1}\tilde{\phi}_i\|_{\ast\ast(|y-\xi_j|> \eta/k)}
    \leq Ck^{1+s-\frac{N}{2}-\frac{N}{q}}.
\end{align}
The discussion for $M_2$ is the same as that for the error term $E$.

For $M_3(\psi)$, we easily see that $\mbox{supp} M_3(\psi)\subset EX.$
By the definition of $N$ and the mean value theorem, there exist $ s, t\in(0,1)$ such that
\begin{align*}
  &\left|N\left(\sum_{j=1}^k \tilde{\phi}_j+\psi\right)\right|  \\
  &=\left|\left|U_{\ast}+\sum_{j=1}^k \tilde{\phi}_i+\psi \right|^{p-1}\cdot \left(U_{\ast}+\sum_{j=1}^k \tilde{\phi}_i+\psi \right)
  -|U_{\ast}|^{p-1} U_{\ast}-p|U_{\ast}|^{p-1}\left(\sum_{j=1}^k \tilde{\phi}_i+\psi \right)\right|\\
  &=p\left|\left|U_{\ast}+s\left(\sum_{j=1}^k \tilde{\phi}_i+\psi\right) \right|^{p-1}\left(\sum_{j=1}^k \tilde{\phi}_i+\psi \right)
  -|U_{\ast}|^{p-1}\left(\sum_{j=1}^k \tilde{\phi}_i+\psi \right)\right|\\
  &=sp\left|U_{\ast}+st\left(\sum_{j=1}^k \tilde{\phi}_i+\psi\right) \right|^{p-2}
  \left|\sum_{j=1}^k \tilde{\phi}_i+\psi \right|^2.
\end{align*}
In the exterior region, using the discussion of Step 1 in Lemma \ref{l22}, we infer
$$\left|\sum_{j=1}^k \tilde{\phi}_i(y)\right|\leq C\|\phi_1\|_{\ast}U(y)\cdot\left(\sum_{j=1}^k\frac{\mu^{\frac{N-2}{2}}}{|y-\xi_j|^{N-2s}}\right).$$
So
\begin{align*}
  \|M_3(\psi)\|_{\ast\ast} &=\left\|\left(1-\sum_{j=1}^k\zeta_j\right)N\left(\sum_{j=1}^k \tilde{\phi}_i+\psi \right)\right\|_{\ast\ast}  \\
   & \leq C\left\|N\left(\sum_{j=1}^k \tilde{\phi}_i+\psi \right)\right\|_{\ast\ast(EX)}\\
   &\leq C\left\|\left[|U_{\ast}|+\left|\sum_{j=1}^k \tilde{\phi}_i\right|+|\psi| \right]^{p-2}\left[\left|\sum_{j=1}^k \tilde{\phi}_i\right|^2+|\psi|^2 \right]\right\|_{\ast\ast(EX)}\\
   &\leq C\|\phi_1\|_{\ast}^2\left\|U^{p}\cdot \sum_{j=1}^k\frac{\mu^{\frac{N-2}{2}}}{|y-\xi_j|^{N-2s}} \right\|_{\ast\ast(EX)}+C\|\psi\|_{\ast}^2\|U^p\|_{\ast\ast(EX)}\\
    & \leq C\|\phi_1\|_{\ast}^2k^{1+s-\frac{N}{2}-\frac{N}{q}}+C\|\psi\|_{\ast}^2.
\end{align*}
Combining the above estimates with respect to $M_1, M_2$ and $M_3$, one has
\begin{align*}
  \|M(\psi)\|_{\ast\ast} &\leq  \|M_1\|_{\ast\ast}+ \|M_2\|_{\ast\ast}+ \|M_3\|_{\ast\ast} \\
   &   \leq Ck^{1+s-\frac{N}{2}-\frac{N}{q}} +C\|\phi_1\|_{\ast}^2k^{1+s-\frac{N}{2}-\frac{N}{q}}+C\|\psi\|_{\ast}^2 .
\end{align*}
Next we will use Banach fixed point theorem to prove Lemma \ref{l41}. So we must show that $\mathcal{M}$ is a contraction mapping from the small ball
in $X$  to the ball itself.

By the mean value theorem, there exist some $s,t\in (0,1)$ such that
\begin{align*}
  \|M(\psi_1)-M(\psi_2)\|_{\ast\ast} &=  \|M_3(\psi_1)-M_3(\psi_2)\|_{\ast\ast}  \\
   & \leq C \left\|N\left(\sum_{j=1}^k \tilde{\phi}_j+\psi_1\right)-N\left(\sum_{j=1}^k \tilde{\phi}_j+\psi_2\right)\right\|_{\ast\ast(EX)} \\
   &\leq \left\|p\left|U_{\ast}+\sum_{j=1}^k \tilde{\phi}_j +s(\psi_1-\psi_2)\right|^{p-1}(\psi_1-\psi_2)-p|U_{\ast}|^{p-1}(\psi_1-\psi_2)\right\|_{\ast\ast(EX)}\\
   &\leq \left\|\left|U_{\ast}+t\sum_{j=1}^k \tilde{\phi}_j +st(\psi_1-\psi_2)\right|^{p-2}\left|\sum_{j=1}^k \tilde{\phi}_j+s(\psi_1-\psi_2)\right|\cdot|\psi_1-\psi_2|\right\|_{\ast\ast(EX)}\\
   &\leq C(\|\phi_1\|_{\ast}+\|\psi_1-\psi_2\|_{\ast})\|\psi_1-\psi_2\|_{\ast}\cdot\|U^p\|_{\ast\ast(EX)}\\
   &\leq C\|\psi_1-\psi_2\|_{\ast(EX)}.
\end{align*}
Then, we have
\begin{align*}
  \|\mathcal{M}(\psi_1-\psi_2)\| &= \|-T[(V_1+V_2)(\psi_1-\psi_2)+(M(\psi_1-\psi_2))]\|_{\ast} \\
   & \leq C\|((V_1+V_2)(\psi_1-\psi_2)\|_{\ast\ast}+C\|(M(\psi_1)-M(\psi_2))]\|_{\ast\ast}\\
   &\leq  C(k^{1-\frac{N}{q}}+\rho)\|\psi_1-\psi_2\|_{\ast} .
\end{align*}
Let $k_0$ be a  large positive  integer  and   $\rho_0$  be  small enough such that for each $k>k_0$ and $\rho<\rho_0$
\begin{equation*}
              C(k^{1+s-\frac{N}{2}-\frac{N}{q}}+\rho)<1.
\end{equation*}
Finally, we use Banach fixed point theorem to complete proof.
\end{proof}

In the following, we will study the first series in \eqref{e45}:
$$ (-\Delta)^s\tilde{\phi}_j-p|U_{\ast}|^{p-1}\zeta_j\tilde{\phi}_j+\zeta_j\left[-p|U_{\ast}|^{p-1}\psi+E-N\left(\tilde{\phi}_j+\displaystyle\sum_{i\not=j} \tilde{\phi}_i+\psi \right)\right]=0, \ j=1,\cdots,k.$$
In fact, by the assumptions \eqref{e41} and  \eqref{e42}, we can make  the changing of variables to simplify the above equations as a single equation:
\begin{equation}\label{e430}
   (-\Delta)^s\tilde{\phi}_1-p|U_{\ast}|^{p-1}\zeta_1\tilde{\phi}_1+\zeta_1\left[-p|U_{\ast}|^{p-1}\psi+E-N\left(\tilde{\phi}_1+\displaystyle\sum_{i\not=1} \tilde{\phi}_i+\psi \right)\right]=0, \ j=1,\cdots,k.
\end{equation}
For convenience sake, we introduce the following notations:
$$\mathcal{N}(\phi_1):=p(|U_{1}|^{p-1}-|U_{\ast}|^{p-1}\zeta_1)\tilde{\phi}_1+\zeta_1\left[-p|U_{\ast}|^{p-1}\psi+E-N\left(\tilde{\phi}_1+\displaystyle\sum_{i\not=1} \tilde{\phi}_i+\psi \right)\right],$$
$$\tilde{h}:=\zeta_1E+\mathcal{N}(\phi_1).$$
We easily see that $\tilde{h}$ satisfy the following properties:
  \begin{equation}\label{e431}
\tilde{h}(y_1, y_2,  \cdots,y_l, \cdots, y_N)=\tilde{h}(y_1, y_2, \cdots,-y_l, \cdots, y_N), l=2,3,\cdots,N;
\end{equation}

\begin{equation}\label{e432}
\tilde{h}(y)=|y|^{-2s-N}\tilde{h}(|y|^{-2}y).
\end{equation}
Then Eq. \eqref{e430} can be reduced to

\begin{equation}\label{e433}
[(-\Delta)^s-p|U_1|^{p-1}\zeta_1]\tilde{\phi_1}+\tilde{h}=0
\end{equation}
According to the definition of $\mu$,  we see that $\mu$ is  related to $\delta$.  Hence
\begin{equation*}
  c_{N+1}(\delta):=\frac{\int_{\mathbb{R}^N}(\zeta_1E+\mathcal{N}(\phi_1))\tilde{v}_{N+1}}{\int_{\mathbb{R}^N}U_1^{p-1}\tilde{v}_{N+1}^2}
\end{equation*}
is also related to $\delta$. Using translating and scaling, we easily know that Eq. \eqref{e433} is equivalent to
\eqref{e31}.  In order to obtain the unique existence of \eqref{e433}, by the results in Lemma \ref{l31}, we need to verify that
$$\int_{\mathbb{R}^N} \tilde{h}\tilde{v}_l=\int_{\mathbb{R}^N} hv_l=0,\ \forall l=1,2, \cdots, N+1.$$
Moreover, by the definition of $c_{N+1}(\delta)$, we get
$$\int_{\mathbb{R}^N} \tilde{h}\tilde{v}_{N+1}=0 \Leftrightarrow \int_{\mathbb{R}^N} hv_{N+1}=0 \Leftrightarrow c_{N+1}(\delta)=0.$$
On basis of the discussion  in Lemma \ref{l41}, we can also obtain
$$\int_{\mathbb{R}^N} \tilde{h}\tilde{v}_l=\int_{\mathbb{R}^N} hv_l=0,\ \forall l=1,2, \cdots, N.$$
Hence, we only need to prove that for  some $\delta_0$,  $c_{N+1}(\delta_0)=0$.
\begin{lemma}\label{l42}
 The $\int_{\mathbb{R}^N}\tilde{h}\tilde{v}_{N+1}$ can be denoted by the following form
 \begin{equation*}
   \int_{\mathbb{R}^N}\tilde{h}\tilde{v}_{N+1}=  C_N\frac{\delta}{k^{N-2s}} [\delta a_N-1]+\frac{1}{k^{N-s}}\Theta_{k}(\delta).
 \end{equation*}
Here $\Theta_{k}(\delta)$ is continuous  related to  $\delta$ and uniformly bounded as $k\rightarrow \infty$, $C_N=p\int_{\mathbb{R}^N}U^{p-1}v_{N+1}$, with
 the positive number
  \begin{equation*}
   a_{N}= 2^{\frac{N-2s}{2}}\lim_{k\rightarrow\infty}\frac{1}{k^{N-2s}}\sum_{j=2}^k\frac{1}{|\xi_1-\xi_j|^{N-2s}}.
 \end{equation*}
\end{lemma}
Obviously, by the above lemma, it is easy to see that for $\delta$ small enough, $\int_{\mathbb{R}^N}\tilde{h}\tilde{v}_{N+1}<0$, while for $\delta$
large enough,  $\int_{\mathbb{R}^N}\tilde{h}\tilde{v}_{N+1}>0$. So, using the continuity property related to $\delta$, there exists $\delta_0>0$ such that
 $\int_{\mathbb{R}^N}\tilde{h}\tilde{v}_{N+1}=0$. The proof of this lemma is similar to Claim 1-4 in \cite{r2}.
\begin{proof}
Note that $\tilde{h}:=\zeta_1E+\mathcal{N}(\phi_1)=E+(\zeta_1-1)E+\mathcal{N}(\phi_1)$. Then we will  take three steps to discuss these terms $E$, $(\zeta_1-1)E$ and $\mathcal{N}(\phi_1)$, respectively.

\noindent\textbf{Step 1}: We  will estimate the term $\int_{\mathbb{R}^N} E\tilde{v}_{N+1}$.  Let $\eta>0$ be a small number, independent of $k$. We set
\begin{equation}\label{eq1}
  \int_{\mathbb{R}^N} E \tilde{v}_{N+1} = \int_{B_1} E \tilde{v}_{N+1} +\int_{\mathbb{R}^N\setminus \cup_{j=1}^k B_j}E\tilde{v}_{N+1}+\sum_{j\not =1}\int_{ B_j}E\tilde{v}_{N+1}.
\end{equation}
where $B_j=B(\xi_j, \frac{\eta}{k})$, $\tilde{v}_{N+1}(y):=\mu^{-\frac{n-2s}{2}}v_{n+1}(\mu^{-1}(y-\xi_1))$.

Let us consider the first term in \eqref{eq1}. By changing the variables $x=\mu y+\xi_1$, we obtain
 $$\int_{B_1} E\tilde{v}_{n+1}=\int_{B(0, \frac{\eta}{\mu k})}\tilde{E}_1v_{n+1}(y)dy,$$
where $$\tilde{E}_1(y)=\mu^{\frac{n+2s}{2}} E(\xi_1+\mu y).$$
In the region $|y|\leq \frac{\eta}{\mu k}$, using  the expansion  \eqref{e28}, we have
\begin{align}\label{eq2}
  \int_{B(0, \frac{\eta}{\mu k})}\tilde{E}_1v_{N+1}(y)dy & = -p\sum_{j\not=1}   \int_{B(0, \frac{\eta}{\mu k})} U^{p-1}U(y-\mu^{-1}(\xi_j-\xi_1))v_{N+1}\nonumber \\
   &\ \ \ \ + p\mu^{\frac{n-2s}{2}}  \int_{B(0, \frac{\eta}{\mu k})} U^{p-1}U(\xi_1+\mu y)v_{N+1}dy\nonumber\\
   & \ \ \ \  +p\int_{B(0, \frac{\eta}{\mu k})}\left[(U+sV)^{p-1}-U^{p-1}\right]Vv_{n+1} dy  \nonumber\\
   &\ \ \ \   +\sum_{j\not=1}   \int_{B(0, \frac{\eta}{\mu k})} U^p(y-\mu^{-1}(\xi_j-\xi_1))v_{N+1}  \nonumber\\
   &\ \ \ \  -\mu^{\frac{N+2s}{2}} \int_{B(0, \frac{\eta}{\mu k})} U^p(\xi_j+\mu y) v_{N+1},
\end{align}
where $$V=\left(-\sum_{j\not=1} U(y-\mu^{-1}(\xi_j-\xi_1)+\mu^{\frac{N-2s}{2}}U(\xi_1+\mu y))\right).$$
We  see that, using the Taylor expansion, for $j\not= 1$,
$$\int_{B(0, \frac{\eta}{\mu k})} U^{p-1}U(y-\mu^{-1}(\xi_j-\xi_1))v_{N+1}=2^{\frac{N-2s}{2}}
C_1\mu^{N-2s}\frac{1}{|\hat{\xi_j}-\hat{\xi_1}|^{N-2s}}(1+(\mu k)^2\Theta_k(\delta)),$$
where $C_N=\int_{\mathbb{R}^N}U^{p-1}v_{N+1}$ and $\hat{\xi}_1=(1,0,\cdots, 0)$ and $\hat{\xi}_j=e^{\frac{2\pi(j-1)}{k}}\hat{\xi}_1$. Moreover,

$$\mu^{\frac{N-2s}{2}}\int_{B(0, \frac{\eta}{\mu k})} U^{p-1}U(\xi_1+\mu y)v_{N+1}dy=C_1\mu^{\frac{N-2s}{2}}(1+(\mu k)^2)\Theta_k(\delta)).$$
For the third term in \eqref{eq2}, using  the inequality $|(a+b)^s-a^s|\leq C|b|^s$, we have
\begin{align}\label{eq3}
  &\int_{B(0, \frac{\eta}{\mu k})}\left[(U+sV)^{p-1}-U^{p-1}\right]Vv_{n+1} dy\nonumber\\ &\leq \sum_{j\not=1}\left| \int_{B(0, \frac{\eta}{\mu k})} U^{p-1}(y-\mu^{-1}(\xi_j-\xi_1))v_{N+1}\right|\nonumber\\
  & \leq C \sum_{j\not=1}\frac{\mu^{N+2s}}{|\hat{\xi_j}-\hat{\xi_1}|^{N-2s}}\int_{B(0, \frac{\eta}{\mu k})}\frac{1}{(1+|y|)^{N-2s}}\nonumber\\
  &\leq  C(\mu k)^{-2s} \sum_{j\not=1}\frac{\mu^{N+2s}}{|\hat{\xi_j}-\hat{\xi_1}|^{N-2s}}.
\end{align}
For the last term in \eqref{eq2}, we estimate
$$\left|\mu^{\frac{N+2s}{2}}\int_{B(0, \frac{\eta}{\mu k})} U^{p}(\xi_j+\mu y)v_{N+1}dy\right|\leq C
\mu^{\frac{N+2s}{2}} \int_{B(0, \frac{\eta}{\mu k})}\frac{1}{(1+|y|)^{N-2s}}\leq C\mu^{\frac{N-2s}{2}}k^{-2s}.$$
Now for the second term in \eqref{eq1}, by H\"{o}lder inequality and estimate of error term, we  get
\begin{align*}
  \int_{EX}E\tilde{v}_{N+1} \leq & C\|(1+|y|)^{N+2s-\frac{2N}{q}}E\|_{L^q(EX)} \cdot \|(1+|y|)^{-N-2s-\frac{2N}{q}}\tilde{v}_{N+1}\|_{L^{\frac{q}{q-1}}(EX)}\\
  \leq &C\|(1+|y|)^{N+2s-\frac{2N}{q}}E\|_{L^q(EX)} \cdot \mu^{\frac{N-2s}{2}}\left(\int_{EX} \left[\frac{|y-\xi_1|^{2s-N}}{(1+|y|)^{N+2s-\frac{2N}{q}})}\right]^{\frac{q}{q-1}}\right)^{\frac{q-1}{q}}\\
\leq&   Ck^{-2N-1+4s}.
\end{align*}
Now let us consider the last term in $\eqref{eq1}$. Set $\tilde{E}_j=\mu^{\frac{N+2s}{2}}E(\xi_j+\mu y), j\not=1$. Performing the change of variables $x=\mu y+\xi_j$.
\begin{align*}
  \left|\int_{ B_j}E\tilde{v}_{N+1}\right| & =\left|\mu^{\frac{N-2s}{2}}\int_{B(0,\eta/(\mu k))}\tilde{E}_j\tilde{v}_{N+1}(\mu y+\xi_j)dy\right| \\
   & \leq C\mu^{\frac{N-2s}{2}}\|(1+|y|)^{N+2s-\frac{2N}{q}}\tilde{E}_j\|_{L^q(B(0,\eta/(\mu k)))}\\
   &\ \ \times\|(1+|y|)^{-N-2s+\frac{2N}{q}}\mu^{-\frac{N-2s}{2}}v_{N+1}(y+\mu^{-1}(\xi_j-\xi_1))\|_{L^{\frac{q}{q-1}}(B(0,\eta/(\mu k)))}\\
   &\leq C\mu^{\frac{N-2s}{2}} \mu^{\frac{N-2s}{2}}(\mu k)^{-N+2s+\frac{N}{q}}\cdot \frac{\mu^{\frac{N-2s}{2}}}{|\xi_j-\xi_1|^{N-2s}}
   \left(\int_1^{\eta/(\mu k)}\frac{t^{N-1}dt}{t^{(N+2s-\frac{2N}{q})\frac{q}{q-1}}}\right)^{\frac{q-1}{q}}\\
   &\leq   C\mu^{N-2s} (\mu k)^{-N+2s+\frac{N}{q}}\frac{\mu^{\frac{N-2s}{2}}}{|\xi_j-\xi_1|^{N-2s}}(\mu k)^{2s-\frac{N}{q}}.
\end{align*}
So we conclude that
\begin{equation*}
  \left|\sum_{j\not =1}\int_{ B_j}E\tilde{v}_{N+1}\right| \leq \frac{\mu^{\frac{N-2s}{2}}}{(\mu k)^{N-4s}}\left[\mu^{N-2s}\sum_{j\not=1} \frac{1}{|\xi_j-\xi_1|^{N-2s}}\right]\leq \frac{Ck^{-2N+s}}{|\xi_j-\xi_1|^{N-2s}}.
\end{equation*}

\noindent\textbf{Step 2}: For  $(\zeta_1-1)E$, we observe that
\begin{align*}
  \left|\int_{\mathbb{R}^N}(\zeta_1-1)E\tilde{v}_{N+1}\right| & \leq C \left|\int_{|x-\xi_1|>\eta/k}E\tilde{v}_{N+1}\right|\\
   & =C \left|\int_{EX}E\tilde{v}_{N+1}\right|+C \sum_{j\not=1}\left|\int_{|x-\xi_j|<\eta/k}E\tilde{v}_{N+1}\right|
\end{align*}
In the exterior region $EX$, by \eqref{e5}, we observe
$$|E(y)|\leq C\frac{\mu^{\frac{N-2s}{2}}}{(1+|y|^2)^{2s}} \sum_{j=1}^{k}\frac{1}{|y-\xi_j|^{N-2s}},$$
where $C$ is a positive constant, independent of $k$. Moreover, in the exterior region, one has
 $$\tilde{v}_{N+1}\leq C\frac{\mu^{\frac{N-2s}{2}}}{|x-\xi_1|^{N-2}}.$$
  So  we easily see  that
$$\left|\int_{EX}E\tilde{v}_{N+1}\right| \leq Ck\mu^{N-2s} \int_{\eta/k}^1 \frac{t^{N-1}}{t^{2N-4s}}dt\leq Ck\mu^{N-2s}k^{N-4s}$$
and
 \begin{equation}\label{eq4}
   \left|\int_{EX}E\tilde{v}_{N+1}\right| =
 \frac{1}{k^{2N-2s-1}}\Theta_{k}(\delta).
 \end{equation}
On the other hand, by changing the variables, $\mu y=x-\xi_j$, we have
$$\int_{|x-\xi_j|<\eta/k}E\tilde{v}_{N+1}=\mu^{\frac{N+2s}{2}}\int_{|y|\leq \eta/(k\mu)} E(\xi_j+\mu y)v_{N+1}(y+\mu^{-1}(\xi_j-\xi_1)).$$
 Note that the argument of Step 2 in Lemma  \ref{l22} implies that
 $$\tilde{E}_j=\mu^{\frac{N+2s}{2}}E(\xi_j+\mu y)\leq C\frac{\mu^{\frac{N-2s}{2}}}{1+|y|^{4s}}.$$
Furthermore, in this region,  we easily obtain
$$|v_{N+1}(y+\mu^{-1}(\xi_j-\xi_1))|\leq C\frac{\mu^{N-2s}k^{N-2s}}{|j-1|^{N-2s}}.$$
Hence, we have
$$\sum_{j\not=1}\left|\int_{|x-\xi_j|<\eta/k}E\tilde{v}_{N+1}\right|\leq k
\mu^{\frac{N-2s}{2}}(k\mu)^{N-2s}\int_{|y|<\eta/(k\mu)}\frac{\mu^{\frac{N-2s}{2}}}{1+|y|^{4s}}dy\leq C k^{-3N+2s+1}.$$

\noindent\textbf{Step 3}:  By the change of variable $x=\mu y+\xi_1$, we have
\begin{align*}
  \int_{\mathbb{R}^N}\mathcal{N}(\phi_1)\tilde{v}_{N+1}dx & = \int_{\mathbb{R}^N}\mathcal{N}(\phi_1)\mu^{-\frac{N-2s}{2}}v_{N+1}(\mu^{-1}(x-\xi_1))dx \\
   & = \int_{\mathbb{R}^N}\mathcal{N}(\phi_1)(\mu y+\xi_1)\mu^{\frac{N+2s}{2}}v_{N+1}(y)dx \\
   &\leq C\|\mu^{\frac{N+2s}{2}}\mathcal{N}(\phi_1)(\mu y+\xi_1) \|_{\ast\ast} \left(\int_{\mathbb{R}^N} \frac{1}{(1+|y|)^{2N}}\right)^{\frac{q-1}{q}}.
\end{align*}
Using the estimates of $f_1, f_2, f_3$ and $f_4$ in Lemma \ref{l43}, we have
\begin{align*}
  &\|\mu^{\frac{N+2s}{2}}\mathcal{N}(\phi_1)(\mu y+\xi_1) \|_{\ast\ast} \left(\int_{\mathbb{R}^N} \frac{1}{(1+|y|)^{2N}}\right)^{\frac{q-1}{q}}\\
   &\leq  Ck^{-\frac{3N}{2}-\frac{N}{q}-5s}(k^{1+s-\frac{N}{2}-\frac{N}{q}}+\|\phi_1\|_{\ast}^2),
\end{align*}
where $C$ is a positive constant, independent of $k$. Recalling that
$\mu=\delta^{\frac{2}{N-2s}}k^{-3}$
and  combining the obtained estimates, we complete the proof.
\end{proof}

\begin{lemma}\label{l43}
For $\tilde{h}$ given above, suppose that
$$h(y):=\mu^{\frac{N+2s}{2}}\tilde{h}(\xi_1+\mu y)$$
satisfies  $\|h\|_{\ast\ast}<\infty$. Then \eqref{e433} has a unique solution $\tilde{\phi}:=\tilde{T}(\tilde{h})$ which satisfies the properties \eqref{e42} and \eqref{e43}  and
$$\int_{\mathbb{R}^N}\phi U^{p-1}v_{N+1}=0,\ \mbox{with}\ \|\phi\|_{\ast}\leq C\|h\|_{\ast\ast},$$
where $\phi(y)=\mu^{\frac{N-2s}{2}}\tilde{\phi}(\xi_1+\mu y).$
\end{lemma}
\begin{proof}
  By Lemma \ref{l42}, we get
  $$\int_{\mathbb{R}^N} hv_{N+1}=0.$$
  Now the oddness of $v_2,v_3,\cdots,v_N $ and the evenness of $h$ imply
  $$\int_{\mathbb{R}^N} hv_i=0,\ \forall i=2,3, \cdots, N.$$
  Next, we only need to show that $\int_{\mathbb{R}^N} hv_1=0$.  Set
  $$I(t):=\int_{\mathbb{R}^N} w_{\mu}(y-t\xi_1)\tilde{h}(y)dy,$$
 where  $$w_{\mu}(y)=\mu^{-\frac{N-2s}{2}}U\left(\mu^{-1}y\right)$$
Then by taking  the derivative of $I(t)$, we have
\begin{align}\label{e434}
  I'(1) &=-\int_{\mathbb{R}^N} \frac{\xi_1}{\mu}\cdot \nabla U_1\left(\mu^{-1}(y-t\xi_1)\right) h\left(\mu^{-1}(y-\xi_1)\right)\cdot \mu^{-\frac{N-2s}{2}} \cdot \mu^{-\frac{N+2s}{2}}dy \bigg|_{t=1}.\nonumber\\
   & =-\frac{\sqrt{1-\mu^2}}{\mu}\int_{\mathbb{R}^N}  U_1(y)h(y)dy\nonumber\\
   &=-\frac{\sqrt{1-\mu^2}}{\mu}\int_{\mathbb{R}^N} hv_1.
\end{align}
  By  making a  transformation $y\rightarrow z=\frac{y}{|y|^2}$, we obtain
\begin{align}\label{e435}
  I(t):&=\int_{\mathbb{R}^N} w_{\mu}(y-t\xi_1)|y|^{-(N+2s)}\tilde{h}(|y|^2y)dy \nonumber \\
  & = \int_{\mathbb{R}^N} w_{\mu}(|y|^{-2}y-t\xi_1)|z|^{-(N-2s)}\tilde{h}(z)dz\nonumber \\
  &= \int_{\mathbb{R}^N}\left(\frac{\mu}{\mu^2+t^2|\xi_1|^2}\right)^{\frac{N-2s}{2}}
   c_{n,s}^{\frac{N-2s}{4s}}\cdot \left[\left|y-\frac{t\xi_1}{\mu^2+t^2|\xi_1|^2}\right|^2
   +\frac{\mu^2}{(\mu^2+t^2|\xi_1|^2)^2}\right]^{-\frac{N-2s}{2}}\tilde{h}(y)dy\nonumber \\
   &= \int_{\mathbb{R}^N} w_{\mu(t)}(y-s(t)\xi_1)\tilde{h}(y)dy,
\end{align}
where $$\mu(t)= \frac{\mu}{\mu^2+t^2|\xi_1|^2},\ s(t)=\frac{t}{\mu^2+t^2|\xi_1|^2}.$$
Through taking  the derivative  on both side of \eqref{e435}, one has

\begin{align}\label{e436}
  I'(1)&=\left[\int_{\mathbb{R}^N} \mu'(t)\frac{\partial(w_{\mu}(y-s(t)\xi_1))}{\partial{\mu}}\tilde{h}(y)dy
  -\xi_1^1\int_{\mathbb{R}^N} s'(t)\frac{\partial( w_{\mu(t)}(y-s(t)\xi_1))}{\partial{y_1}}\tilde{h}(y)dy\right]\Bigg|_{t=1} \nonumber\\
  &=2\mu^2 \int_{\mathbb{R}^N} \tilde{v}_{N+1}(y)\tilde{h}(y)dy-\frac{(2\mu^2-1)\sqrt{1-\mu^2}}{\mu}\int_{\mathbb{R}^N} hv_1\nonumber\\
  &=\frac{(2\mu^2-1)\sqrt{1-\mu^2}}{\mu}\int_{\mathbb{R}^N} hv_1.
 \end{align}
So \eqref{e434} and \eqref{e436} imply that
$$I'(1)=-\frac{\sqrt{1-\mu^2}}{\mu}\int_{\mathbb{R}^N} hv_1=\frac{(2\mu^2-1)\sqrt{1-\mu^2}}{\mu}\int_{\mathbb{R}^N} hv_1.$$
Obviously, this  shows that $$\int_{\mathbb{R}^N} hv_1=0.$$

Now Lemma \ref{l31} implies that Eq. \eqref{e433} has   a unique solution $\tilde{\phi}:=\tilde{T}(\tilde{h})$ which is even with respect to each of the variables
$y_2, y_3,\cdots, y_{N}$ satisfying
$$\|\tilde{\phi}\|_{\ast}=\|\tilde{T}(\tilde{h})\|_{\ast}\leq C\|\tilde{h}\|_{\ast\ast}.$$
But $\tilde{h}$ in Eq. \eqref{e433} does not satisfy  the property \eqref{e415}, so the discussion of  Lemma \ref{l41} can not be applied to obtain the existence directly. We must  prove that the operator $\tilde{T}$ is a contraction mapping again.

In the following, for convenience  sake, the term $\zeta_1E+\mathcal{N}(\phi_1)$  is spilt  into five terms. And we estimate these terms one by one.
Set
$$f_1:=p\zeta_1(U_1^{p-1}-|U_{\ast}|^{p-1})\cdot \tilde{\phi}_1,\ \ \ f_2:=p(1-\zeta_1)U_1^{p-1}\tilde{\phi}_1,$$
$$f_3:=-p\zeta_1U_{\ast}^{p-1}\psi(\phi_1),\ \  \  f_4:=\xi_1N(\sum_{j=1}^k\tilde{\phi}_j+\psi(\phi_1)),\  \ \ f_5:=\zeta_1E,$$
and $$\tilde{f}_i(y)=\mu^{\frac{N+2s}{2}}f_i(\xi_1+\mu y), \ \  i=1,\cdots,5.$$
Then $$\tilde{h}=\sum_{i=1}^5 f_i.$$
Due to the cut-off function $\zeta_1$, we see that
$$\mbox{supp} f_j\subset \{y: |y-\xi_1|<\eta/k\}=:IN_1\subset IN,\ \ j=1,3,4,5.$$
For $f_1$, we get
\begin{align*}
  |\tilde{f}_1| &\leq \left|p\left|U(y)+\sum_{j=2}^kU(y+\mu^{-1}(\xi_1-\xi_j))-
  \mu^{\frac{N-2s}{2}}U(\xi_1+\mu y)\right|^{p-1}-pU^{p-1}(y)\right|\cdot |\phi_1(y)| \\
   & \leq  C\left|\sum_{j=2}^kU(y+\mu^{-1}(\xi_1-\xi_j))+
  \mu^{\frac{N-2s}{2}}U(\xi_1+\mu y)+U(y)\right|^{p-2}\\
 & \ \ \ \cdot |\mu^{\frac{N-2s}{2}}U(\xi_1+\mu y)+U(y)|\cdot U(y)\cdot \|\phi_1\|_{\ast} \\
  &\leq CU^{p-1}(y)  \mu^{\frac{N-2s}{2}}\|\phi_1\|_{\ast}  \leq C\frac{ \mu^{\frac{N-2s}{2}}}{1+|y|^{4s}}\|\phi_1\|_{\ast}.
\end{align*}
Hence we take the the same argument of Step 2 in Lemma \ref{l22} and infer
\begin{equation}\label{e437}
   \|f_1\|_{\ast\ast}=\|f_1\|_{\ast\ast(IN_1)}\leq
              C \|\phi_1\|_{\ast}k^{1+s-\frac{N}{2}-\frac{N}{q}}.
\end{equation}
For $f_2$, we see that
$$|\tilde{f}_2(y)|=|\zeta_1(\mu y+\xi_1)-1|\cdot U^{p-1}\cdot |\phi_1|\leq C|U|^p\|\phi_1\|_{\ast}.$$
Thus,  \begin{align}\label{u1}
         \|f_2\|_{\ast\ast} & \leq C\left[\int_{|y-\xi_1|>\eta/k}(1+|y|)^{(N+2s)q-2N}\mu^{-\frac{N+2s}{2}q}
         \left|\tilde{f}_2^q\left(\frac{y-\xi_1}{\mu}\right)\right|dy\right]^{1/q}
          \nonumber\\
        & \leq C\left[\mu^{\frac{q(N+2s)}{2}}\cdot\left( \int_{\eta/k}^{1} r^{N-1-(N-2s)pq}dr+\int_{1}^{\infty} r^{(N+2s)q-2N-(N-2s)pq+N-1}dr\right)\right]^{1/q}\|\phi_1\|_{\ast} \nonumber\\
        &\leq C\mu^{ \frac{N+2s}{2}} k^{(N+2s)-\frac{N}{q}}\|\phi_1\|_{\ast}+C\mu^{ \frac{q(N+2s)}{2}}\|\phi_1\|_{\ast} <Ck^{ -\frac{3(N+2s)}{2}}\|\phi_1\|_{\ast}.
      \end{align}

Analogously, applying the estimate of $\psi$ in Lemma \ref{l41}, we have
\begin{align*}
  |\tilde{f_3}|& \leq CU^{p-1}\mu^{\frac{N-2s}{2}} \|\psi(\phi_1)\|_{\infty} \\
   & \leq CU^{p-1}\mu^{\frac{N-2s}{2}} \|\psi(\phi_1)\|_{\ast}\\
   &\leq C\mu^{\frac{N-2s}{2}}\frac{1}{1+|y|^{4s}} (k^{1+s-\frac{N}{2}-\frac{N}{q}}+\|\phi_1\|_{\ast}^2).
\end{align*}
and
\begin{equation}\label{e438}
\|f_3\|_{\ast\ast} \leq
               Ck^{-\frac{N}{q}-2s}(k^{1+s-\frac{N}{2}-\frac{N}{q}}+\|\phi_1\|_{\ast}^2).
\end{equation}
Now, for $f_4$, noting that
$$\tilde{N}=|V_{\ast}+\hat{\phi}|^{p-1}(V_{\ast}+\hat{\phi})-|V_{\ast}|^{p-1}V_{\ast}-p|V_{\ast}|^{p-1}\hat{\phi}.$$
where $\hat{\phi}(y):=\mu^{\frac{N-2s}{2}}\phi (\xi_1+\mu y)$, and
$$V_{\ast}(y)=U(y)+\sum_{j=2}^kU(y+\mu^{-1}(\xi_1-\xi_j))-\mu^{\frac{N-2s}{2}}U(\xi_1+\mu y).$$
So, for $\phi=\sum_{j=1}^k\tilde{\phi}_j+\psi(\phi_1)$, we have
$$|\tilde{f}_4(y)|\leq C U^{p-1}\mu^{\frac{N-2s}{2}}\left[\|\phi_1\|_{\ast}+(\|\phi_1\|_{\ast}^2+k^{1+s-\frac{N}{2}-\frac{N}{q}})\right],$$
thus
\begin{equation}\label{e439}
\|f_4\|_{\ast\ast} \leq
               Ck^{-\frac{N}{q}-2s}\left[\|\phi_1\|_{\ast}+(\|\phi_1\|_{\ast}^2+k^{1+s-\frac{N}{2}-\frac{N}{q}})\right].
\end{equation}
It follows from the estimate of the error term $E$ that
\begin{equation}\label{e440}
\|f_5\|_{\ast\ast} \leq
               Ck^{1+s-\frac{N}{2}-\frac{N}{q}}.
\end{equation}
Combining the obtained estimates for $f_1, \cdots, f_5$, for all $\hat{\phi}, \hat{\phi}_1,  \hat{\phi}_2\in B_{\rho}(0)\subset X$, we have

\begin{equation}\label{e443}
\|\mathcal{M}(\hat{\phi})\|_{\ast}\leq C\sum_{i=1}^{5}\|f_{i=1}(\hat{\phi})\|_{\ast\ast} \leq
               Ck^{-\frac{N}{q}-2s}(\|\phi_1\|_{\ast}+\|\hat{\phi}\|_{\ast}^2),
\end{equation}
and
\begin{align*}
 \|\mathcal{M}(\hat{\phi_1})-\mathcal{M}(\hat{\phi_2})\|_{\ast}&\leq C\sum_{i=1}^5\|f_i(\hat{\phi}_1)-f_i(\hat{\phi}_2)\|_{\ast\ast}  \\
& \leq  Ck^{-\frac{N}{q}-2s}(\|\hat{\phi}_1\|_{\ast}+\|\hat{\phi}_2\|_{\ast})\|\hat{\phi}_1-\hat{\phi}_2\|_{\ast} \\
&=: \lambda\|\hat{\phi}_1-\hat{\phi}_2\|_{\ast}, \  \mbox{ with}\ \ \lambda<1.
\end{align*}
Hence M is a contraction mapping from $B_\rho(0)$ to $ B_\rho(0)$, for $k$ large enough. By
the Banach fixed point theorem, there exists a unique solution $\tilde{\phi}_1$ of Eq. (\ref{e433}).
\end{proof}
\section{Proof of main result } %
%首先，我们假设原问题具有形式解 $U=U_{\ast}+\phi$, 那么我们将原问题化成了问题\eqref{e3}. 通过引入截断函数， 我们又假设
%$\phi=\sum_{j=1}^k \tilde{\phi}_j+\psi$, 那么问题\eqref{e3} 又变成了系统(\ref{e45})。在引理\ref{l41} 和 引理\ref{l43}中，
%for all $k>k_0$, 我们用 Banach fixed theorem 分别获得了 $\tilde{\phi_j}$和$\psi$的存在性。 从而，对于原问题我们获得了
%sign-changing solution $u_k=U_{\ast}+ \phi_j+\psi$.
We assume that the  original problem (\ref{P2}) admits a solution of the form:  $$U=U_{\ast}(y)+\phi(y).$$
Then problem (\ref{P2}) is  converted into  Eq. \eqref{e3}.
Through introducing the cut-off functions, and assuming that $\phi=\sum_{j=1}^k \tilde{\phi}_j+\psi$, Eq. \eqref{e3} is turned
 into a system of equations of $\tilde{\phi}_j, j=1,2,\cdots,k,$ and $\psi$ (see \eqref{e45}). Hence, we only need to prove
  the existences of  $\tilde{\phi}_j, j=1,2,\cdot,k,$ and $\psi$, which are done in Section 4 by Banach fixed point theorem. So for $k>k_0$, the sign-changing solutions
 $u_k=U_{\ast}(y)+\sum_{j=1}^k \tilde{\phi}_j+\psi$  for problem(\ref{P2})   are obtained.

In short, the outline of our proofs  is as follows:
\begin{align*}
  &\mbox{Eq.}(\ref{P2}) \xLeftrightarrow{u=U_{\ast}+\phi}
   \mbox{Eq.}(\ref{e3})  \xLeftrightarrow{\phi=\sum_{j=1}^k \tilde{\phi}_j+\psi}\mbox{Eq.}(\ref{e45})
    \xLeftrightarrow{\quad}\\
   &\left\{\begin{array}{l}
   \mbox{Eq.}(\ref{e46})\  \ \mbox{unique existence} \Leftarrow \mbox{Banach fixed point theorem;}  \\
 \mbox{Eq.}(\ref{e433})\ \   \mbox{unique existence} \Leftarrow \mbox{Banach fixed point theorem.}
  \end{array}
   \right.
\end{align*}

The proof is completed.

 \end{document}